\documentclass[12 pt]{amsart}
\usepackage{amsfonts,amsmath,amsthm,amssymb} 

\usepackage[utf8]{inputenc}
\usepackage{enumerate}

\theoremstyle{plain}
\newtheorem{theorem}{Theorem}[section]
\newtheorem{corollary}[theorem]{Corollary}
\newtheorem{lemma}[theorem]{Lemma}

\newtheorem{problem}[theorem]{Problem}
\newtheorem{oq}[theorem]{Open question}

\theoremstyle{definition}

\theoremstyle{remark}

\newtheorem{example}{Example}[section]

\numberwithin{equation}{section}

\begin{document}
\title[Invariance for means]{Explicit solutions of the invariance equation for means}

\author[J. Matkowski]{Janusz Matkowski}
\address{Faculty of Mathematics, Computer Sciences and Econometrics, University of Zielona Gora, Szafrana 4A, 65-516 Zielona Gora, Poland}
\email{J.Matkowski@wmie.uz.zgora.pl}
\author[M. Nowicka]{Monika Nowicka}
\address{Institute of Mathematics and Physics, UTP University of Science and Technology, al. prof. Kaliskiego 7, 85-796 Bydgoszcz, Poland}
\email{monika.nowicka@utp.edu.pl}
\author[A. Witkowski]{Alfred Witkowski}
\address{Institute of Mathematics and Physics, UTP University of Science and Technology, al. prof. Kaliskiego 7, 85-796 Bydgoszcz, Poland}
\email{alfred.witkowski@utp.edu.pl}
\date{16.12.2014}

\begin{abstract}
Extending the notion of projective means we first generalize an invariance
identity related to the Carlson log given in \cite{KM}, and then, more generally,
given a bivariate symmetric, homogeneous and monotone mean $M$, we give explicit formula for
a rich family of pairs of $M$-complementary means. We prove
that this method cannot be extended for higher dimension. Some
examples are given and two open questions are proposed.

\end{abstract}
\maketitle
\section{Introduction}

A function $\Phi :X\rightarrow Y$ is called \textit{invariant} with respect
to a selfmap $T:X\rightarrow X$ (briefly, $T$-\textit{invariant}) if $\Phi
\circ T=\Phi $. Invariant functions appear in iteration theory and fixed
point theory. For instance, if $X$ is a metric space, $T$ is  continuous
and the sequence $\left( T^{n}\right) _{n\in \mathbb{N}}$ of iterates of $T$
is pointwise convergent, then the function $\Phi (x)=\lim_{n\rightarrow
\infty }T^{n}(x)$ is $T$-invariant. A model illustration offer the mean-type
mappings, i.e. the mappings of the form $\left( K,L\right) ,$ where the
coordinate functions $K,L$ are bivariate means. Some conditions guarantying
convergence of iterates $\left( K,L\right) ^{n}$ to a unique $\left(
K,L\right) $-invariant mean-type mapping $\left( M,M\right) ,$ and $M\circ
\left( K,L\right) =M$\ (\cite{JM2013}, also \cite{JM-ECIT2009,JM99,AW}),
generalize in particular, the well-known theorem of Gauss \cite{Gauss} on
the arithmetic-geometric iterations. If the invariance equality $M\circ
\left( K,L\right) =M$ is satisfies one says that the means $K$ and $L$\ are
mutually $M$-\textit{complementary} with respect to $M$ (briefly, 
 $M$-\textit{complementary}) (\cite{JM-AeqMath99}).

It happens quite exceptionally, when a given mean-type mapping $\left(
K,L\right) $ one can find the explicit form of the $\left( K,L\right) $%
-invariant mean. The identity $\mathsf{G}\circ \left( \mathsf{A},\mathsf{H}%
\right) =\mathsf{G}$, where $\mathsf{A}$, $\mathsf{G}$, $\mathsf{H}$ denote,
respectively, arithmetic, geometric and harmonic mean, (equivalent to the
classical Pythagorean harmony proportion $\frac{\mathsf{A}}{\mathsf{G}}=%
\frac{\mathsf{G}}{\mathsf{H}})$, meaning that $\mathsf{G}$ is $\left( \mathsf{%
A},\mathsf{H}\right) $-invariant, and allowing to conclude that 
\begin{equation*}
\lim_{n\rightarrow \infty }\left( A,H\right) ^{n}\left( x,y\right) =\left(
G\left( x,y\right) ,G\left( x,y\right) \right) \quad x,y>0,
\end{equation*}%
is an example.

In these circumstances it is natural to ask if, a given mean $M,$ one can
find effectively some nontrivial pairs of $M$-complementary means. This
problem appeared in connection with \cite{KM} where the authors proved that
for $t\in \left[ -1,1\right] ,$ $t\neq 0,$ the unsymmetric means 
\begin{equation*}
K_{t}=tx^{t}\frac{x-y}{x^{t}-y^{t}}\qquad \text{and}\qquad L_{t}=ty^{t}\frac{%
x-y}{x^{t}-y^{t}}
\end{equation*}%
are mutually complementary with respect to the logarithmic mean $\mathsf{L}$%
, i.e. they satisfy the invariance equation 
\begin{equation}
\mathsf{L\circ }(K_{t},L_{t})=\mathsf{L}.  \label{eq:L-invariance}
\end{equation}

In section \ref{Generalized projective means}, we extend the notion of projective means by creating  $2^{\mathfrak{c}}$ pairs of means $(\mathsf{P}_A,\mathsf{P}_{A'})$ satisfying $\{\mathsf{P}_A(x,y),\mathsf{P}_{A'}(x,y)\}=\{x,y\}$, so that the functions

\begin{align*}
K_{t,A}(x,y) &=t\mathsf{P}_{A}^t(x,y)\frac{x-y}{x^{t}-y^{t}}\\ 
\intertext{and}
 L_{t,A}(x,y) &=t\mathsf{P}_{A'}^t(x,y)\frac{x-y}{x^{t}-y^{t}},
\end{align*}
are $\mathsf{L}$-complementary means. We also give necessary and sufficient conditions for these means to be symmetric and/or homogeneous, thus we provide answer to a question posed in \cite{KM}.

In section \ref{The L is not enough}, following the ideas of \cite{KM} and \cite{AW1}, we prove that if $M:\mathbb{R}%
_{+}^{2}\rightarrow \mathbb{R}_{+}$ is a monotone, homogeneous and symmetric
mean, then for every $t\in \left( -1,1\right) $ the function  $M_{t}$ given by 
\begin{equation*}
M_{t}( x,y) =\left( \frac{M( x,y) }{M(
x^{t},y^{t}) }\right) ^{\frac{1}{1-t}}
\end{equation*}%
is a mean and the functions 
\begin{equation*}
K_{t}(x,y)=x^{t} M_{t}^{1-t}(x,y) \quad \text{and}\quad
L_{t}(x,y)=y^{t}M_{t}^{1-t}(x,y) 
\end{equation*}%
are homogeneous $M$-complementary means. Actually, a stronger results holds true, namely the functions $%
K_{t,A}=\mathsf{P}_{A}^{t} M_{t}^{1-t}$
and $L_{t,A}=\mathsf{P}_{A'}^{t}M_{t}^{1-t}$ are homogeneous $M$-complementary means. The
construction the means $K_{t},L_{t}$ depends on $M$, and they inherit the
assumed symmetry and homogeneity of $M$, but in general, the monotonicity is
not hereditary.

In section \ref{sec:one more step} we ask whether the projective means considered in previous sections can be replaced by some other means.  The examples  constructed show that in even in case of very classical means $M,C,D$ the function $N_{t}=\left( \frac{M}{%
M\circ (C^{t},D^{t})}\right) ^{\frac{1}{1-t}}$ may not be a mean. Nevertheless, we prove rather surprising fact, that if $M$ is a symmetric, homogeneous, monotone mean, $C$, $D$ are
arbitrary means and $t\in \left( 0,1\right) $, then   the functions $%
K_{t}=C^{t}N_{t}^{1-t} $, \  $L_{t}=D^{t}N_{t}^{1-t}$ are $M$-complementary means. 

Thus, the applied method gives explicit formulas for complementary means in case of monotone, symmetric and homogeneous means $M$;
moreover the means $K_{t}$ and $L_{t}$ inherit the symmetry and/or
homogeneity from $C$ and $D$. \ Homogeneity of $M$ is crucial here. We conclude that section with two open questions,  and one of them is whether the
monotonicity condition can be relaxed?

Noting that all the presented results have their translative
counterparts, we formulate in section \ref{Remark on translative means} the translative counterpart of Theorem \ref{thm:main}.  An application of this result for the arithmetic mean $\mathsf{A}$ gives all possible pairs of $\mathsf{A}$-complementary weighted arithmetic means. 

\section{Preliminaries}
The set of positive real numbers is denoted by $\mathbb{R}_+$.
A \textit{mean} is a function $M\colon\mathbb{R}_+^2\to\mathbb{R}_+$ satisfying
\begin{equation}
\min \{x,y\}\leq M(x,y)\leq \max\{x,y\}.
\label{eq:definition of mean}
\end{equation}

Depending on additional properties a mean is called
\begin{description}
	\item[strict] if the inequalities in \eqref{eq:definition of mean} are strict whenever $x\neq y$, 
	\item[symmetric] if $M(x,y)=M(y,x)$ for all $x,y$,
	\item[monotone] if $M(x_1,y_1)\leq M(x_2,y_2)$ for $x_1\leq x_2,\ y_1\leq y_2$,
	\item[homogeneous] if $M(\lambda x,\lambda y)=\lambda M(x,y)$ for all $x,y,\lambda>0$.
\end{description}
Note that since there are no decreasing means, a ''monotone mean'' means in fact an
''increasing mean''.

Classical means will be denoted by sans-serif capital letters. Thus
\begin{align*}
	\mathsf{A}(x,y)&=\dfrac{x+y}{2}, & \mathsf{G}(x,y)&=\sqrt{xy}, & \mathsf{H}(x,y)&=\dfrac{2xy}{x+y},\\
	\mathsf{L}(x,y)&=\dfrac{x-y}{\log x-\log y},&\mathsf{P}_1(x,y)&=x,&\mathsf{P}_2(x,y)&=y
\end{align*}
denote respectively the arithmetic, geometric, harmonic, logarithmic and the two projective means. For historical reason the exception is made for minimum and maximum means that are denoted by $\mathsf{min}$ and $\mathsf{max}$.

If $F\colon\mathbb{R}_{+}^{2}\rightarrow \mathbb{R}_{+}$  is homogeneous
(of the order $1$), then $F(x,y)=yF(\frac{x}{y},1)$ for all $x,y>0$, so its
values are uniquely determined by the \textit{trace} function%
\begin{equation*}
f:=F(\cdot ,1).
\end{equation*}%
If $F$ is also symmetric, then the identity $F(x,1)=xF(x^{-1},1)$ shows that it is uniquely determined by the
restriction of the trace function $f$ to any of the intervals $\left[
1,\infty \right) $ or $\left( 0,1\right] .$ \\
A homogeneous function $M$ is a mean if, and only if, 
\begin{equation*}
0<\frac{m\left( x\right) -1}{x-1}\leq 1,\text{ \ \ \ \ \ }x>0,\text{ }x\neq
1;\text{ \ \ \ \ \ \ \ \ \ }m\left( 1\right) =1;
\end{equation*}%
moreover, if the trace function $m$ is increasing and $m\left( 1\right) =1$,
then $M$ is a mean.

The trace function will be denoted by corresponding lowercase letter.
\section{Generalized projective means}\label{Generalized projective means}
In this section we construct $2^\mathfrak{c}$ pairs of symmetric means satisfying \eqref{eq:L-invariance}. 

Let $\mathbb{X}= \mathbb{R}_+^2\setminus\{(x,x): x\in\mathbb{R}_+\}$.
 For $A\subset \mathbb{X}$ we define the \textit{generalized projective mean} $\mathsf{P}_A\colon\mathbb{R}_+^2\to \mathbb{R}_+$ by
\begin{equation*}
\mathsf{P}_A(x,y)=\begin{cases}
	x &	(x,y)\in A,\\
	y & (x,y) \not\in A
\end{cases}.
\end{equation*}
A set $A\subset \mathbb{X}$ is called \textit{asymmetric} if
$$(x,y)\in A\Leftrightarrow (y,x)\notin A\quad\text{for}\quad x\neq y.$$
We denote $A'=\mathbb{X}\setminus A$.

Note the following elementary properties of generalized projective means:
\begin{enumerate}[1)]
	\item $\mathsf{P}_1=\mathsf{P}_\mathbb{X}$, $\mathsf{P}_2=\mathsf{P}_{\emptyset}$,
	\item $\mathsf{min}=\mathsf{P}_{\{(x,y):x<y\}}$, $\mathsf{max}=\mathsf{P}_{\{(x,y):x>y\}}$,
	\item $\{x,y\}=\{\mathsf{P}_A(x,y),\mathsf{P}_{A'}(x,y)\}$,\label{P_A complementary}
	\item $\mathsf{P}_A$ is symmetric if and only if $A$ is asymmetric,\label{P_A symmetric}
	\item $\mathsf{P}_A$ is homogeneous if and only if $A$ is a positive cone, i.e. $A=\lambda A$ for all $\lambda>0$.\label{P_A homogeneous}
\end{enumerate}

The property \ref{P_A complementary}) implies that we can replace the means in \eqref{eq:L-invariance} by
$$K_{t,A}=t\mathsf{P}_{A}^t(x,y)\frac{x-y}{x^t-y^t}\quad\text{and}\quad L_{t,A}=t\mathsf{P}_{A'}^t(x,y)\frac{x-y}{x^t-y^t}$$
preserving the invariance property. Playing with parameter $A$ according to properties \ref{P_A symmetric} and \ref{P_A homogeneous}, one obtains symmetric or/and homogeneous solutions. Later we will construct a lot of other solutions to the invariance problem.

\section{The $\mathsf{L}$ is not enough}\label{The L is not enough}
In this section we follow the ideas developed in \cite{KM} to obtain complementary means  for other means than the logarithmic one. Given symmetric, homogeneous mean $M$ we shall seek for a pair of means of the form $x^tM_t^{1-t}(x,y),\ y^tM_t^{1-t}(x,y)$, where $-1<t<1$. Let us give it a try.
\begin{equation*}
M(x^tM_t^{1-t}(x,y),y^tM_t^{1-t}(x,y))=M_t^{1-t}(x,y)M(x^t,y^t)=M(x,y).
\end{equation*}
Solving this equation for $M_t$ we obtain
\begin{equation}
M_t(x,y)=\left(\frac{M(x,y)}{M(x^t,y^t)}\right)^\frac{1}{1-t}.
\label{eq:def M_t}
\end{equation}
If $M_t$ happens to be a mean and $t>0$, then obviously both $\mathsf{P}_1^{t}M_t^{1-t}$ and $\mathsf{P}_2^{t}M_t^{1-t}$ are means and satisfy $M\circ(\mathsf{P}_1^{t}M_t^{1-t},\mathsf{P}_2^{t}M_t^{1-t})=M.$ They remain means for $t<0$, due to  the identity $x^tM_t^{1-t}=y^{-t}M_{-t}^{1+t}$.\\
The following theorem gives a simple criterion for $M_t$ to be means.
\begin{lemma}
	If $M$ is a homogeneous, symmetric mean, then the conditions are equivalent:
	\begin{enumerate}[a)]
		\item For every real $-1<t<1$ the function
		$$M_t(x,y)=\left(\frac{M(x,y)}{M(x^t,y^t)}\right)^\frac{1}{1-t}$$
		is a mean.
		\item $M$ is monotone.
	\end{enumerate}
\end{lemma}
\begin{proof}
	Suppose $M$ is monotone. Then for $x<y$ we have
	$$x^{1-t}=\frac{M(x^t x^{1-t},y^t x^{1-t})}{M(x^t,y^t)}\leq \frac{M(x,y)}{M(x^t,y^t)} \leq \frac{M(x^t y^{1-t},y^t y^{1-t})}{M(x^t,y^t)}=y^{1-t},$$
	so $M_t$ is a mean.
	
	Assume a) holds. Since $M$ is homogeneous and symmetric, it suffices to show that its trace $m$ increases.  If $x\leq 1\leq y$, then setting $t=0$ we see that $m(x)\leq 1\leq m(y)$.
	If $1<x<y$ then there exists $0<t<1$ such that $x=y^t$ and then
	$$1\leq m_t^{1-t}(y)=\frac{m(y)}{m(x)},$$
	which concludes the proof, since the case $x<y<1$ is similar.
\end{proof}
So we have proved the following fact.
\begin{theorem}\label{thm:xy invariance}
	If $M$ is a monotone, homogeneous and symmetric mean and for $-1<t<1$ the functions $M_t$ are given by \eqref{eq:def M_t}, then
	$$K_t(x,y)=x^tM_t^{1-t}(x,y)\quad\text{and}\quad L_t(x,y)=y^tM_t^{1-t}(x,y)$$
	are homogeneous means and satisfy
	$$M\circ (K_t,L_t)=M.$$
\end{theorem}

Arguing as in the previous section we obtain the corollary.
\begin{corollary}
	If $M$ is a monotone, symmetric and homogeneous mean, and $M_t$ is given by formula \eqref{eq:def M_t}, then for any generalized projective mean $\mathsf{P}_A$ the functions
	$$K_{t,A}(x,y)=\mathsf{P}_A^t(x,y)M_t^{1-t}(x,y)\quad\text{and}\quad L_{t,A}(x,y)=\mathsf{P}_{A'}^t(x,y)M_t^{1-t}(x,y)$$
	are homogeneous means and satisfy
	$$M\circ (K_{t,A},L_{t,A})=M.$$
\end{corollary}

Note that the means $K_t$ and $L_t$ inherit the symmetry and homogeneity, but in general, the monotonicity is not hereditary. 

\begin{theorem}
	If $M\neq \mathsf{max}, K_t$ and $L_t$ are as in Theorem \ref{thm:xy invariance} and $$\lim_{x\to 0+} m(x)>0,$$ then the means $K_t$ and $L_t$ are not monotone if $t>0$.
\end{theorem}
\begin{proof}
	Consider the mean $L_t$. One has
	$$\lim_{x\to 0+}l_t(x)=\lim_{x\to 0+}m_t^{1-t}(x)=1=l_t(1),$$
	which shows that $l_t$ is not monotone. Case $K_t$ is similar.
\end{proof}
Note that in case of the logarithmic mean the functions obtained are weighted geometric means of a projective mean and a Stolarsky mean $\mathsf{STO}_{1,t}$. The Stolarsky means defined (in general case) by 
$$\mathsf{STO}_{r,s}(x,y)=\left(\frac{s}{r}\frac{x^r-y^r}{x^s-y^s}\right)^{\frac{1}{r-s}}=
\left(\frac{\mathsf{L}(x^r,y^r)}{\mathsf{L}(x^s,y^s)}\right)^{\frac{1}{r-s}}$$
are monotone, and therefore the resulting means \eqref{eq:L-invariance} are so. But the Stolarsky means contain a group of means for which the resulting invariant means lack monotonicity. These are all $\mathsf{STO}_{r,s}$ with $r,s>0$, in particular the arithmetic mean $\mathsf{A}=\mathsf{STO}_{2,1}$ and the generalization of the Heronian means
$$\mathsf{STO}_{1+\frac{1}{n},\frac{1}{n}}(x,y)=\frac{x+x^{\frac{n-1}{n}}y^{\frac{1}{n}}+\dots+x^{\frac{1}{n}}y^{\frac{n-1}{n}}+y}{n+1}.$$
All of them have a non-zero limit at zero, thus the resulting means are not monotone.

The reader will verify that if $M\circ (K,L)=M$ and $M(0+,1)>0$ and $K(0+,1)=0$, then $L$ is not monotone.
\section{One more step towards invariance}\label{sec:one more step}
The generalized projective means are a kind of extremities in the world of means. It is natural to ask whether similar solution can be obtained for arbitrary means. Let us formulate the problem as follows.
\begin{problem}
	Suppose $M$ is a symmetric, homogeneous and increasing mean and $t>0$. Characterize the means $C,D$ for which there exists a bivariate function $N$ such that both	$C^tN^{1-t}${and} $D^tN^{1-t}$ are means and 
	\begin{equation*}
	M\circ (C^tN^{1-t},D^tN^{1-t})=M.
	\end{equation*}
\end{problem}
One can easily calculate that
\begin{equation}
N(x,y)=N_t(x,y)=\left(\frac{M(x,y)}{M(C^t(x,y),D^t(x,y))}\right)^\frac{1}{1-t}.
\label{eq:def N}
\end{equation}
The answer depends very much on the means involved. Here are some examples.
\begin{example}
	For $t<\frac{1}{2}$ and arbitrary means $C,D$ 
	$$\left(\frac{\mathsf{G}}{\mathsf{G}\circ(C^t,D^t)}\right)^\frac{1}{1-t}$$
	is a mean. 
	
	Indeed, since $\frac{\mathsf{G}^2}{K}$ is a mean for arbitrary mean $K$, and $\mathsf{G}^{\frac{1}{t}}\circ(C^t,D^t)$ is a mean, we can write
	$$\left(\frac{\mathsf{G}}{\mathsf{G}\circ(C^t,D^t)}\right)^\frac{1}{1-t}=
	\left(\frac{\mathsf{G}^2}{\mathsf{G}^{1/t}\circ(C^t,D^t)}\right)^\frac{t}{1-t}\mathsf{G}^{1-\frac{t}{1-t}},$$
	so the left-hand side is a weighted geometric mean of means.
\end{example}
\begin{example}\label{ex:A}
	For arbitrary $0<t<1$ the function
	$$\left(\frac{\mathsf{A}(x,y)}{\mathsf{A}(\mathsf{A}^t(x,y),\mathsf{H}^t(x,y))}\right)^\frac{1}{1-t}$$
	is not a mean.
	
	Suppose for $x>1$ the inequality
	$$\left(\frac{\mathsf{a}(x)}{\mathsf{A}(\mathsf{a}^t(x),\mathsf{h}^t(x))}\right)^\frac{1}{1-t}\leq x.$$
	is valid. This is equivalent to
	$$\frac{\mathsf{a}(x)}{x}\leq \frac{\mathsf{a}^t(x)+\mathsf{h}^t(x)}{2x^t}.$$
	As $x$ tends to infinity, the left-hand side tends to $\frac{1}{2}$, while the limit of the right-hand side is $\frac{1}{2^{1+t}}$, so the assumption was wrong.
\end{example}	
This example can be generalized. 
\begin{example}\label{ex:MKL not means}
If $M, K, L$ are homogeneous means satisfying
$$\lim_{x\to 0+} m(x)>0,\quad\lim_{x\to\infty}\frac{l(x)}{k(x)}=0,\quad\lim_{x\to\infty} \frac{k(x)}{x}<1,$$
and $M$ is symmetric, then for any $0<t<1$ the corresponding function is not a mean, because 
\begin{align*}
\lim_{x\to\infty}\frac{1}{x}
\left(\frac{m(x)}{M(k^t(x),l^t(x))}\right)^{\frac{1}{1-t}}	
&	=\lim_{x\to\infty}
\frac{1}{x}\left(\frac{xm\left(\frac{1}{x}\right)}{k^t(x)m\left(\frac{l^t(x)}{k^t(x)}\right)}\right)^{\frac{1}{1-t}}
\\
	&	=\lim_{x\to\infty}\left(\frac{x}{k(x)}\right)^{\frac{t}{1-t}}>1.
\end{align*}
\end{example}
\bigskip
One can easily verify, that for $p>0$ the functions $\mathsf{STO}_{1,p}, \mathsf{A}$ and $\mathsf{H}$ satisfy the conditions in Example \ref{ex:MKL not means}. \\
Given the fact, that 
	$\lim_{p\to 0+} \mathsf{STO}_{1,p}=\mathsf{L}$ pointwise, 	the next example looks a little bit surprising.

\begin{example}
		The function
		$$N(x,y)=\left(\frac{\mathsf{L}(x,y)}{\mathsf{L}(\mathsf{A}^{{1}/{2}}(x,y),\mathsf{H}^{{1}/{2}}(x,y))}\right)^2$$
		is a mean.
		\bigskip
		
		To show this we need some quite elementary facts. First note
	\begin{equation}
	1<\frac{\mathsf{a}(x)}{\mathsf{h}(x)}<x\quad\text{for}\quad x>1.
	\label{ineq: A/H<x}
	\end{equation}
	By the convexity of  $\sinh$ in $\mathbb{R}_+$, the divided difference $\frac{\sinh u}{u}$ increases. Hence, taking into account that
	\begin{equation}
	\mathsf{L}(x,x^{-1})=\frac{\sinh \log x}{\log x}
	\label{eq:function g}
	\end{equation}
	we conclude that the function $(1,\infty)\ni x\mapsto\mathsf{L}(x,x^{-1})$ is positive and increases.
	
	\bigskip
	To show that $N$ is a mean it is enough to prove that for $x>1$ the inequalities $1<N(x,1)=n(x)<x$ hold. The left one is valid, because $\sqrt{\mathsf{a}(x)\mathsf{h}(x)}=\sqrt{x}$, \eqref{eq:function g} and \eqref{ineq: A/H<x} yield
	\begin{align*}
		\sqrt{n(x)}&=\frac{\mathsf{l}(x)}{\mathsf{L}(\sqrt{\mathsf{a}(x)},\sqrt{\mathsf{h}(x)})}
	\\
	&=\frac{\sqrt{x}\,\mathsf{L}\!\left(\sqrt{x},\frac{1}{\sqrt{x}}\right)}{\sqrt[4]{\mathsf{a}(x)\mathsf{h}(x)}\,\mathsf{L}\!\left(\sqrt[4]{\frac{\mathsf{a}(x)}{\mathsf{h}(x)}},\sqrt[4]{\frac{\mathsf{h}(x)}{\mathsf{a}(x)}}\right)}>\sqrt[4]{x}>1.
	\end{align*}
To 	show the other inequality remind that the power mean of order $1/2$ 
$$\mathsf{A}_{1/2}(x,y)=\left(\frac{\sqrt{x}+\sqrt{y}}{2}\right)^2=\left(\frac{\mathsf{L}(x,y)}{\mathsf{L}\!\left({\sqrt{x}},\sqrt{y}\right)}\right)^2$$ satisfies  $\mathsf{A}_{1/2}\leq \mathsf{A}$. Then we use \eqref{ineq: A/H<x} and the monotonicity of the logarithmic mean to obtain
\begin{align*}
	\sqrt{n(x)}&=\frac{\mathsf{l}(x)}{\mathsf{L}(\sqrt{\mathsf{a}(x)},\sqrt{\mathsf{h}(x)})}=
	\frac{1}{\sqrt{\mathsf{a}(x)}}\frac{\mathsf{l}(x)}{\mathsf{l}\!\left(\sqrt{\frac{\mathsf{h}(x)}{\mathsf{a}(x)}}\right)}\\
	&<\frac{1}{\sqrt{\mathsf{a}(x)}}\frac{\mathsf{l}(x)}{\mathsf{l}\!\left(\frac{1}{\sqrt{x}}\right)}=	\sqrt{\frac{\mathsf{a}_{1/2}(x)}{\mathsf{a}(x)}}\sqrt{x}<\sqrt{x}.
\end{align*}

	\end{example}
	The above examples show that to decide whether a particular function is a mean might be quite complicated. But in fact, we do not need that much. Fortunately we can prove that the functions $C^tN^{1-t}$ and $D^tN^{1-t}$ are means. 
	\begin{theorem}\label{thm:main}
		If $M$ is a symmetric, homogeneous, increasing mean and $C$ and $D$ are arbitrary means then for all $0<t<1$ the functions 
		$$K_t(x,y)=C^t(x,y)N_t^{1-t}(x,y)\quad\text{and}\quad L_t(x,y)=D^t(x,y)N_t^{1-t}(x,y),$$
		where $N_t$ is given by \eqref{eq:def N}, are means and satisfy the invariance equation
		$$M(K_t(x,y),L_t(x,y))=M(x,y).$$
	\end{theorem}
	\begin{proof}
		Take arbitrary $x,y>0,\ x<y$. Our goal is to show that 
	\begin{equation*}
	x\leq C^t(x,y)N_t^{1-t}(x,y)\leq y \quad\text{ and }\quad x\leq D^t(x,y)N_t^{1-t}(x,y)\leq y.
	\end{equation*}
	We shall prove the left inequalities, the proof of the right ones being similar.
	Note firstly that 
	$$C^t(x,y)N_t^{1-t}(x,y)=\frac{M(x,y)}{M\!\left(\frac{D^t(x,y)}{C^t(x,y)},1\right)}$$
	and consider two cases:\\
	\textsc{Case} $C(x,y)\leq D(x,y)$.\\
	In this case
	$$1\geq\frac{D^t(x,y)}{C^t(x,y)}\leq \frac{D(x,y)}{C(x,y)}\leq \frac{y}{x}$$ 
	and
	$$M\!\left(\frac{D^t(x,y)}{C^t(x,y)},1\right)>1,$$
so
$$y\geq M(x,y)\geq \frac{M(x,y)}{M\!\left(\frac{D^t(x,y)}{C^t(x,y)},1\right)}\geq 
\frac{M(x,y)}{M\!\left(\frac{y}{x},1\right)}=x$$	
	\textsc{Case} $C(x,y)> D(x,y).$\\
	Now 
	$$1\geq \frac{D^t(x,y)}{C^t(x,y)}\geq \frac{D(x,y)}{C(x,y)}\geq \frac{x}{y}$$ 
	and
	$$M\!\left(\frac{D^t(x,y)}{C^t(x,y)},1\right)<1,$$
so
$$x\leq M(x,y)\leq \frac{M(x,y)}{M\!\left(\frac{D^t(x,y)}{C^t(x,y)},1\right)}\leq 
\frac{M(x,y)}{M\!\left(\frac{x}{y},1\right)}=y.$$	
\end{proof}
Clearly, the means $K_t$ and $L_t$ inherit the symmetry and/or homogeneity from $C$ and $D$. The discussion on monotonicity from the previous section applies here as well.

\bigskip
The method described above works very well in case of monotone, symmetric and homogeneous means $M$. Clearly, homogeneity is crucial here, but one can ask whether the monotonicity condition can be relaxed?
\begin{oq}
Consider  a symmetric and homogeneous mean $M$. Do there exists three functions $K,L,N$ and a real number $0<t<1$ such that the functions $K^tN^{1-t}$ and $L^tN^{1-t}$ are means and the equality
$$M(K^tN^{1-t},L^tN^{1-t})=M$$
holds?
\end{oq}
The examples in Section \ref{sec:one more step} show that the next question may be challenging as well.
\begin{oq}
Suppose $M$ is a symmetric, homogeneous and increasing mean. Do there exist non-trivial means $C, D$ and a real number $0<t<1$ such that the function
$$N(x,y)=\left(\frac{M(x,y)}{M(C^t(x,y),D^t(x,y))}\right)^{\frac{1}{1-t}}$$
is a mean?
\end{oq}

\section{Means of $n$ variables}\label{sec:n variables}
It is natural to ask, whether Theorem \ref{thm:main} can be extended to higher dimension.
Assume then $n>2$ and let $M,C_1, \ldots, C_n: \mathbb{R}^n_+\rightarrow\mathbb{R}_+$ be  means with  $M$  symmetric, homogeneous and increasing.
For arbitrary $t\in(0,1)$ the  invariance equation 
$$M\circ(C_1^tN_t^{1-t}, \ldots, C_n^tN_t^{1-t})=M$$ 
can be solved in the same way as in case of means of two variables.
So we obtain
\begin{align*}
N_t=\left(\frac{M}{M\circ(C_1^t, \ldots, C_n^t)}\right)^{\frac{1}{1-t}}.
\end{align*}

The example below shows the answer to our question is negative. 
\begin{example}
Take $M=C_1=\mathsf{A},\ C_2=\dots=C_n=\mathsf{G}$ and let $t\in(0,1)$ be arbitrary. Then the function
$$K_t=\mathsf{A}^t\frac{\mathsf{A}}{\mathsf{A}\circ(\mathsf{A}^t,\mathsf{G}^t,\dots,\mathsf{G}^t)}$$
is not a mean.

Indeed, setting  $x_1=1$ and  $x_2=\ldots=x_n=x$ we obtain
\begin{align*}
\lim_{x\rightarrow\infty}&\frac{K_t(1, x,\ldots, x)}{\mathsf{max}(1, x,\ldots, x)} \\ 
&=\lim_{x\rightarrow\infty} \frac{\left(\frac{1+(n-1)x}{n}\right)^t}{x} \frac{1+(n-1)x}{\left(\frac{1+(n-1)x}{n}\right)^t+(n-1)\left(\sqrt[n]{x^{n-1}}\right)^t} \\ 
&= \lim_{x\rightarrow\infty} \frac{\frac{1}{x} +(n-1)}{1+(n-1)\left(\frac{x\displaystyle^{n-1}}{\left(\frac{1}{n}\left(1+(n-1)x\right)\right)\displaystyle^n}\right)^{\frac{t}{n}}}=n-1>1.
\end{align*}

\end{example}

We want to emphasize that even if the functions $C_i^t(\mathbf{x}) N^{1-t}_t(\mathbf{x})$ may not be means,  they form an $n$-tuple of invariant functions - a fact that can be useful in some applications.
\section{Remark on translative means}\label{Remark on translative means}

A mean $N:\mathbb{R}^{2}\rightarrow \mathbb{R}$ is called \textit{translative%
} if $N\left( x+\tau ,y+\tau \right) =N\left( x,y\right) +\tau $ for all $%
\tau ,x,y\in \mathbb{R}$. \ 

Recall that a bivariate mean on $\mathbb{R}^{2}$ is both homogeneous and
tranlative if, and only if, it is a weighted arithmetic mean \cite[Theorem 1, p. 234]{Aczel}. This fact explains great popularity of the arithmetic means.

\ 

Since a mean $M:\mathbb{R}_{+}^{2}\rightarrow \mathbb{R}_{+}$ is homogeneous
if, and only if, the mean $N:\mathbb{R}^{2}\rightarrow \mathbb{R}$ defined by%
\begin{equation*}
N\left( x,y\right) :=\log M\left( \exp x,\exp y\right) ,
\end{equation*}%
is translative, all the notions, results and questions posed above have
their "translative" counterparts. In particular Theorem \ref{thm:xy
invariance} can be reformulated as follows.

\begin{theorem}
If $N:\mathbb{R}^{2}\mathbb{\rightarrow R}$ is a monotone, translative and
symmetric mean and for $-1<t<1$, the functions $N_{t}$ are given by 
\begin{equation*}
N_{t}(x,y)=\frac{N\left( x,y\right) -N\left( tx,ty\right) }{1-t},\text{ \
\ \ \ \ }x,y\in \mathbb{R}\text{,}
\end{equation*}%
then the functions 
\begin{equation*}
K_{t}(x,y)=tx+\left( 1-t\right) N_{t}(x,y)\quad \text{and}\quad
L_{t}(x,y)=ty+\left( 1-t\right) N_{t}(x,y)
\end{equation*}%
are translative means and satisfy 
\begin{equation*}
N\circ (K_{t},L_{t})=N.
\end{equation*}
\end{theorem}

\begin{example}
Since the arithmetic mean $\mathsf{A}$ is monotone, translative and
symmetric, applying this result we conclude that $\mathsf{A}\circ
(K_{t},L_{t})=\mathsf{A},$ that $\mathsf{A}$ is invariant with respect to the
mean-type mapping $(K_{t},L_{t}),$ where 
\begin{equation*}
K_{t}(x,y)=\frac{1+t}{2}x+\frac{1-t}{2}y\quad \text{and}\quad L_{t}(x,y)=%
\frac{1-t}{2}x+\frac{1+t}{2}
\end{equation*}%
for $x,y\in \mathbb{R}.$

It is not difficult to check that $\left\{ (K_{t},L_{t}):t\in \left(
-1,1\right) \right\} $ is a family of all weighted arithmetic mean-type
mapping such that $\mathsf{A}\circ (K_{t},L_{t})=\mathsf{A}$.
\end{example}


\bigskip

\begin{thebibliography}{9}
\bibitem{Aczel} J. Acz\'{e}l, \textit{Lectures on functional equations and
their applications}, Mathematics in Science and Engineering 19, New York:
Academic Press,1966.

\bibitem{Gauss} C. F. Gauss, \textit{Bestimmung der Anziehung eines
elliptischen Ringen} \textit{Nach lass zur Theorie des
arithmetisch-geometrischen Mittles und der Modul-funktion}, Akad. Verlag.
M.B.H. Leipzig, New York, 1927.

\bibitem{KM} P. Kahlig, J. Matkowski, \textit{Logarithmic complementary
means and an extension of Carlson's log}, MIA (accepted).

\bibitem{JM99} J. Matkowski, \textit{Iterations of mean-type mappings and
invariant means}, Ann. Math. Silesianae \textbf{13} (1999), 211-226.

\bibitem{JM-AeqMath99} -, \textit{Invariant and complementary means},
Aequationes Math. \textbf{57}(1999), 87-107.

\bibitem{JM-ECIT2009} -, \textit{Iterations of the mean-type mappings, }%
(ECIT' 08), (A.N. Sharkovsky, I.M. Susko (Eds.), Grazer Math. Ber., Bericht
Nr. \textbf{354} (2009), 158-179.

\bibitem{JM2013} \textit{Iterations of the mean-type mappings and uniqueness
of invariant means}, Annales Univ. Sci. Budapest., Sect. Comp. \textbf{41}%
(2013), 145-158.

\bibitem{AW} A. Witkowski, \textit{On Seiffert-like means,} JMI (submitted).

\bibitem{AW1} A. Witkowski, \textit{On two- and four-parameter family,} RGMIA, Vol \textbf{12}, no 1, 3
\end{thebibliography}
\end{document}